\documentclass[11pt]{amsart}
\usepackage[latin1]{inputenc}
\usepackage{graphicx}
\usepackage[T1]{fontenc}
\usepackage{amssymb,pb-diagram}
\usepackage{amsmath}
\usepackage{amsthm,amssymb,amsfonts}
\usepackage{amssymb,amscd}
\usepackage{epic,eepic,epsfig}
\usepackage{latexsym}
\usepackage{stmaryrd}
\usepackage{xy}
\usepackage{stmaryrd}
\usepackage{amssymb,amsthm,enumerate,color,mathrsfs}
\xyoption{all}

\usepackage{amssymb,amsthm,enumerate,color,mathrsfs}

\newtheorem{Theorem}{Theorem}[section]

\newtheorem{Proposition}[Theorem]{Proposition}

\newtheorem{Corollary}[Theorem]{Corollary}

\theoremstyle{definition}
\newtheorem{Definition}[Theorem]{Definition}
\newtheorem{Remark}[Theorem]{Remark}

\newtheorem{Example}[Theorem]{Example}
\newtheorem{Examples}[Theorem]{Examples}

\theoremstyle{definition}

\theoremstyle{remark}

\newcommand{\bib}{\bibitem}

\begin{document}

\renewcommand{\proofname}{Proof}

\title{Banach spaces with the Blum-Hanson Property}

\author{F. Netillard}
\address{Universit\'{e} Bourgogne Franche-Comt\'{e}, Laboratoire de Math\'{e}matiques de Besan\c{c}on, UMR 6623,
16 route de Gray, 25030 Besan\c{c}on Cedex, FRANCE.}
\email{francois.netillard2@univ-fcomte.fr}




\maketitle





\begin{abstract}
We are interested in a sufficient condition given in \cite{L} to obtain the Blum-Hanson property and we then partially answer two questions asked in this same article on other possible conditions to have this property for a separable Banach space.
\end{abstract}

\markboth{}{}




\section{Introduction}

These notes are essentially inspired by article \cite{L} in which sufficient new conditions to justify that a Banach space has the Blum-Hanson property were obtained.\\
We recall that, for a (real or complex) Banach space $X$, and a contraction T on X ($T$ is a bounded operator on $X$ with $\|T\|\leq 1$), we say that $T$ has the \textit{Blum-Hanson property} if, for $x,y \in X$ such that $T^nx$ weakly converges to $y \in X$ when $n$ tends to infinity, the mean
\begin{center}$\dfrac{1}{N}\displaystyle\sum\limits_{k=1}^NT^{n_k}x$\end{center}
tends toward $y$ in norm for any increasing sequence of integers $(n_k)_{k \geq 1}$.\\
The space $X$ is said to have the \textit{Blum-Hanson property} if every contraction on $X$ has the Blum-Hanson property.\\
Note, to understand the interest in this property and its historical aspect, that, when $X$ is a Hilbert space and the linear operator $T$ is a contraction, for all $x \in X$ such that $T^nx \overset{w}{\rightarrow} 0$, the arithmetic mean \begin{center}$\dfrac{1}{N}\displaystyle\sum\limits_{k=1}^NT^{n_k}x$\end{center}
is norm convergent to $0$ for any increasing sequence of integers $(n_k)_{k \geq 1}$. This result was first proved by J.R. Blum and D.L. Hanson in \cite{B} for isometries induced by measure-conserving transformations, then in \cite{A} and \cite{J} for arbitrary contractions.\\
The most notable spaces having the Blum-Hanson property are the Hilbert spaces and the $\ell_p$ spaces for $1\leq p<\infty$.\\
Note that this property is not preserved under renormings (see \cite{MT}). This raises the following question: "Which Banach spaces can be renormed to have the Blum-Hanson property~?", already asked in \cite{L}. This question motivated the writing of this article.\\
To understand the main results of this work, we give first the following definition of an asymptotically uniformly smooth norm.

\begin{Definition}
Consider a Banach space $(X,\|\cdot\|)$. By following the definitions due to V. Milman \cite{M} and the notations of \cite{Jo2} and \cite{L3}, for $t \in [0,\infty)$, $x \in S_X$ and $Y$ a closed vector subspace of $X$, we define the modulus of asymptotic uniform smoothness, $\overline{\rho}_X(t)$:
\begin{center}$\overline{\rho}_X(t,x,Y)=\underset{y \in S_Y}{\textnormal{sup}}(\|x+ty\|-1).$\end{center}
Then
\begin{center}$\overline{\rho}_X(t,x)=\underset{Y \in \textnormal{cof}(X)}{\textnormal{inf}}\overline{\rho}_X(t,x,Y) \hspace{1cm}\textnormal{ and }\hspace{1cm} \overline{\rho}_X(t)=\underset{x \in S_X}{\textnormal{sup}}\overline{\rho}_X(t,x)$.\end{center}

\noindent The norm $\|\cdot\|$ is said to be \textit{asymptotically uniformly smooth} (in short AUS) if \begin{center}$\underset{t \to 0}{\textnormal{lim}}\frac{\overline{\rho}_X(t)}{t}=0.$\end{center}
\end{Definition}

Now, we can give the main property of this paper which partially answers the previous question:
\begin{Theorem}\label{C2}
Let $Y$ be a separable Banach space whose norm is AUS.\\
Then $Y$ has an equivalent norm with the Blum-Hanson property.
\end{Theorem}

\begin{Remark}
A Banach space $Y$ which has an AUS norm is an Asplund space. Consequently, $Y$ is separable if and only if its dual is separable.
\end{Remark}

\section{Banach space with property $(m_p)$}

N. Kalton and D. Werner introduced in \cite{K} the property $(m_p)$:
\begin{Definition}
A Banach space $X$ has property $(m_p)$, where $1\leq p\leq\infty$ if, for any $x \in X$ and every weakly null sequence $(x_n)\subset X$, it holds that:
\begin{center}$\underset{n \to \infty}{\textnormal{lim~sup}}\|x+x_n\|=(\|x\|^p+\textnormal{lim~sup}\|x_n\|^p)^{\frac{1}{p}}.$\end{center}
For $p=\infty$, the right-hand side is of course to be interpreted as max$(\|x\|,\textnormal{lim~sup}\|x_n\|)$.\\
\end{Definition}

\begin{Remark}
We shall say that $X$ has property \textit{sub}-$(m_p)$ if we can replace "$=$" by "$\leq$" in the previous definition.
\end{Remark}

\begin{Examples}
\textbullet~$\ell_p$ has property $(m_p)$, $c_0$ has property $m_{\infty}$.\\
\end{Examples}

In \cite{L}, P. Lef\`evre, \'E. Matheron and A. Primot obtained the following property which was a corollary of one of the main theorems of their paper. It is this property which allowed us in particular to obtain Theorem \ref{C2}.

\begin{Proposition}\label{C1}\cite{L}
For any $p\in (1,\infty]$, property sub-$(m_p)$ implies Blum-Hanson property.
\end{Proposition}

\begin{Example}(see \cite{N} and \cite{G})\\
We recall the definition of the James space $J_p$. This is the real Banach space of all sequences $x = (x(n))_{n \in \mathbb{N}}$ of real numbers satisfying $\underset{n \to \infty} {\textnormal{lim}} x(n) = 0$, endowed with the norm
\begin{center}$\|x\|_{J_p}=\textnormal{sup}\left\{\left(\displaystyle\sum\limits_{i=1}^{n-1}|x(p_{i+1})-x(p_i)|^p\right)^{\frac{1}{p}} : 1\leq p_1<p_2<\cdots<p_n\right\}.$\end{center}
This is a quasi-reflexive Banach space which is isomorphic to its bidual.\\
Historically, R.C. James has focused exclusively on $J=J_2$ (see \cite{Ja}), and I.S. Edelstein and B.S. Mityagin (cf \cite{E}) are apparently the first to have observed that we could generalize the definition to $p \geq 1$ arbitrary and to have observed the quasi-reflexivity of $J_p$ for any $p> 1$.\\
There exists an equivalent norm $|\cdot|$ on $J_p$ (see \cite{N}, Corollary 2.4, for the proof) such that, for all $x,y\in J_p$ verifying max $\{i \in \mathbb{N}: x(i) \neq 0 \} < \textnormal{min} \{i \in \mathbb{N}: y(i) \neq 0\} $, it holds that
\begin{center} $|x+y|^p \leq |x|^p + |y|^p$. \end{center}
Thus, $\tilde{J_p} := (J_p,|\cdot|)$ has the sub-$(m_p)$ property, and therefore the Blum-Hanson property.
\end{Example}

We now introduce a notion that is essentially dual to sub-$(m_p)$.

\begin{Definition} Let $X$ be a separable Banach space and $q\in (1,\infty)$. We say that $X^*$ has property sup-$(m_q)^*$ if, for any $x^*\in X^*$ and any weak$^*$-null sequence $(x_n^*)$ in $X^*$, we have:
$$\liminf_{n\to \infty}\|x^*+x^*_n\|^q\ge \|x^*\|^q+\liminf_{n\to \infty}\|x^*_n\|^q.$$
\end{Definition}

The following is an easy adaptation of the proof of Proposition 2.6 from \cite{Go}.

\begin{Proposition}\label{duality} Let $X$ be a separable Banach space. Let $p\in (1,\infty)$ and $q$ be its conjugate exponent. Assume that $X^*$ has property sup-$(m_q)^*$, then $X$ has property sub-$(m_p)$.
\end{Proposition}

\begin{proof} Let $x\in X$ and $(x_n)$ be a weakly null sequence in $X$ and denote $s=\limsup_n\|x_n\|$.  Pick $y_n^* \in X^*$ so that $\|y^*_n\|=1$ and $y^*_n(x+x_n)=\|x+x_n\|$. After extracting a subsequence, we may assume that $(y^*_n)$ is weak$^*$ converging to $x^* \in B_{X^*}$. Denote $x^*_n= y^*_n-x^*$ and assume also, as we may, that $\lim_n\|x^*_n\|=t$. Since $X^*$ has sup-$(m_q)^*$, we have that $\|x^*\|^q+t^q \le 1$. Therefore
\begin{align*}
\limsup_n\|x+x_n\|&=\limsup_n(x^*+x^*_n)(x+x_n)\le x^*(x)+st \\
&\le (\|x^*\|^q+t^q)^{1/q}(\|x\|^p+s^p)^{1/p}\le (\|x\|^p+s^p)^{1/p}.
\end{align*}
This concludes our proof.

\end{proof}

\section{Main results}

We give some definitions that will be used later.

\begin{Definition}
Given an FDD $(E_n)$, $(x_n)$ is said to be a block sequence with respect to $(H_i)$ if there exists a sequence of integers $0=m_1<m_2<\cdots$ such that $x_n \in \bigoplus_{j=m_n}^{m_{n+1}-1}E_j$.
\end{Definition}

\begin{Definition}
Let $1\leq q \leq p \leq \infty$ and $C<\infty$. A (finite or infinite) FDD $(E_i)$ for a Banach space $Z$ is said to satisfy $C-(p,q)$ estimates if for all $n \in \mathbb{N}$ and block sequences $(x_i)_{i=1}^n$ with respect to $(E_i)$:
\begin{center}$C^{-1}\left(\displaystyle\sum\limits_{1}^{n}\|x_i\|^p\right)^{\frac{1}{p}} \leq \left\|\displaystyle\sum\limits_{1}^{n}x_i\right\| \leq C\left(\displaystyle\sum\limits_{1}^{n}\|x_i\|^q\right)^{\frac{1}{q}}.$\end{center}
\end{Definition}

${}$






For the central theorem for this work, we now recall the definition of the Szlenk index.

\begin{Definition}
Let $X$ be a Banach space and $K$ be a weak$^*$-compact subset of $X^*$. For $\epsilon>0$, let $\mathcal{V}$ be the set of all weak$^*$-open subsets of $K$ such that the norm diameter (for the norm of $X^*$) of $V$ is less than $\epsilon$, and \begin{center}$s_{\epsilon}K = K\backslash \bigcup\{V : V \in \mathcal{V}\}.$\end{center}
As a remark, $s_{\epsilon}^{\alpha}B_{X^*}$ is defined inductively for any ordinal $\alpha$ by
\begin{center}$s_{\epsilon}^{\alpha+1}B_{X^*}=s_{\epsilon}(s_{\epsilon}^{\alpha}B_{X^*})$\end{center}
and \begin{center}$s_{\epsilon}^{\alpha}B_{X^*}=\underset{\beta<\alpha}{\bigcap}s_{\epsilon}^{\beta}B_{X^*}\textnormal{ if } \alpha \textnormal{ is a limit ordinal}.$\end{center}
We define $Sz(X,\epsilon)$ to be the least ordinal $\alpha$ so that $s_{\epsilon}^{\alpha}B_{X^*}=\emptyset$ if such an ordinal exists. Otherwise we write $Sz(X,\epsilon)=\infty$ by convention.\\
We will then denote $Sz(X)$ the Szlenk index of $X$, defined by \begin{center}$Sz(X)=\underset{\epsilon>0}{\textnormal{sup}}~Sz(X,\epsilon).$\end{center}
\end{Definition}

\begin{Remark}
For a detailed report about the Szlenk index, one can refer to \cite{L2}.\\
Note that the Szlenk index was introduced by W. Szlenk in \cite{S} to show that there is no universal reflexive space for the class of separable reflexive spaces.
\end{Remark}

\noindent The main ingredient of our argument is the following result, which is deduced  from \cite{Kn} (Corollary 5.3) and is already cited in \cite{L1} (in the proof of Theorem $4.15$). However, in \cite{L1}, we do not find the detailed proof of this property, that we include now.

\begin{Proposition}
Let $Y$ be a separable Banach space such that Sz$(Y) \leq \omega$, where $\omega$ denote the first infinite ordinal.\\
Then, $Y$ can be renormed in such a way that there exists $q\in (1,\infty)$ so that
\begin{center}$\textnormal{lim sup}\|y+y_n\|^q \leq \|y\|^q + \textnormal{lim sup}\|y_n\|^q$,\end{center}
whenever $y\in Y$ and $(y_n)$ is a weakly null sequence in $Y$.
\end{Proposition}

\begin{proof}
According to Corollary 5.3  from  \cite{Kn}, Sz$(Y) \leq \omega$ implies that there exists a Banach space $Z$ with a boundedly complete FDD $(E_i)$ (in particular $Z$ is isometric to a dual space $X^*$) with the following properties.
\begin{enumerate}
    \item There exists $p\in (1,\infty)$ such that $(E_i)$ satisfies $1-(p,1)$ estimates.
    \item $Y^*$ is isomorphic (norm and weak$^*$) to a weak$^*$-closed subspace $F$ of $Z=X^*$.
\end{enumerate}
Let us denote $S:Y^*\to F$ this isomorphism. Then, there exists a subspace $G$ of $X$ such that $G^\perp=F$ and $S$ is the adjoint of an isomorphism $T$ from $X/G$ onto $Y$. Let now $q$ be the conjugate exponent of $p$. It is thus enough to prove that $E=X/G$ has sub-$(m_q)$. Since $X^*=Z$ satisfies $1-(p,1)$ estimates with respect to the boundedly complete FDD $(E_i)$, it is immediate that $X^*$ has sup-$(m_p)^*$. Now this property passes clearly to its weak$^*$-closed subspace $F$ and $E^*$ has sup-$(m_p)^*$. Finally, we deduce from Proposition \ref{duality} that $E$ has sub$-(m_q)$. This finishes our proof.
\end{proof}

\noindent Thanks to this Proposition, we obtain the Theorem \ref{C2}.

\begin{proof}
The proof is immediate by applying the Proposition \ref{C1}.
\end{proof}





${}$

We will now talk about property $(M^*)$.\\
\noindent It was studied by N.J. Kalton and D. Werner in \cite{K} :

\begin{Definition}
A Banach space $X$ has property $(M^*)$ if, for $u^*,v^*\in S_{X^*}$ and $(x_n^*)\subseteq X^*$ a weak$^*$ null sequence, it holds that
\begin{center} $\underset{n}{\textnormal{lim~sup}}\|u^*+x_n^*\|=\underset{n}{\textnormal{lim~sup}}\|v^*+x_n^*\|$.\end{center}
\end{Definition}

\begin{Remark}
It has been shown in \cite{K} that, if $X$ is a separable Banach space having property $(M^*)$, then its dual is separable.
\end{Remark}

\noindent The following Proposition follows from \cite{D} (Proposition 2.2 of this article).

\begin{Proposition}\label{P3}
Let $X$ be a separable Banach space with property $(M^*)$. Then $X$ is asymptotically uniformly smooth for a norm $\|\cdot\|_M$.
\end{Proposition}

\begin{Corollary}
Let $X$ be a separable Banach space with property $(M^*)$. Then $X$ has an equivalent norm with the Blum-Hanson property.
\end{Corollary}

\begin{proof}
It follows from the Proposition \ref{P3} and from the Theorem \ref{C2}.
\end{proof}




\noindent \textbf{Acknowledgment:} The author is grateful to Gilles Lancien and Audrey Fovelle for helpful discussions.

\end{document}